\documentclass[12pt]{article}
\usepackage{latexsym, amssymb, amsthm}
\usepackage{dsfont}

\DeclareMathAlphabet\gothic{U}{euf}{m}{n}

\makeatletter
\def\eqnarray{\stepcounter{equation}\let\@currentlabel=\theequation
\global\@eqnswtrue
\tabskip\@centering\let\\=\@eqncr
$$\halign to \displaywidth\bgroup\hfil\global\@eqcnt\z@
  $\displaystyle\tabskip\z@{##}$&\global\@eqcnt\@ne
  \hfil$\displaystyle{{}##{}}$\hfil
  &\global\@eqcnt\tw@ $\displaystyle{##}$\hfil
  \tabskip\@centering&\llap{##}\tabskip\z@\cr}

\def\endeqnarray{\@@eqncr\egroup
      \global\advance\c@equation\m@ne$$\global\@ignoretrue}

\def\@yeqncr{\@ifnextchar [{\@xeqncr}{\@xeqncr[5pt]}}
\makeatother

\begin{document}
\bibliographystyle{tom}

\newtheorem{lemma}{Lemma}[section]
\newtheorem{thm}[lemma]{Theorem}
\newtheorem{cor}[lemma]{Corollary}
\newtheorem{prop}[lemma]{Proposition}

\theoremstyle{definition}

\newtheorem{remark}[lemma]{Remark}
\newtheorem{exam}[lemma]{Example}
\newtheorem{definition}[lemma]{Definition}

\newcommand{\gota}{\gothic{a}}
\newcommand{\gotb}{\gothic{b}}
\newcommand{\gotc}{\gothic{c}}
\newcommand{\gote}{\gothic{e}}
\newcommand{\gotf}{\gothic{f}}
\newcommand{\gotg}{\gothic{g}}
\newcommand{\gothh}{\gothic{h}}
\newcommand{\gotk}{\gothic{k}}
\newcommand{\gotm}{\gothic{m}}
\newcommand{\gotn}{\gothic{n}}
\newcommand{\gotp}{\gothic{p}}
\newcommand{\gotq}{\gothic{q}}
\newcommand{\gotr}{\gothic{r}}
\newcommand{\gots}{\gothic{s}}
\newcommand{\gott}{\gothic{t}}
\newcommand{\gotu}{\gothic{u}}
\newcommand{\gotv}{\gothic{v}}
\newcommand{\gotw}{\gothic{w}}
\newcommand{\gotz}{\gothic{z}}
\newcommand{\gotA}{\gothic{A}}
\newcommand{\gotB}{\gothic{B}}
\newcommand{\gotG}{\gothic{G}}
\newcommand{\gotL}{\gothic{L}}
\newcommand{\gotS}{\gothic{S}}
\newcommand{\gotT}{\gothic{T}}

\newcounter{teller}
\renewcommand{\theteller}{(\alph{teller})}
\newenvironment{tabel}{\begin{list}%
{\rm  (\alph{teller})\hfill}{\usecounter{teller} \leftmargin=1.1cm
\labelwidth=1.1cm \labelsep=0cm \parsep=0cm}
                      }{\end{list}}

\newcounter{tellerr}
\renewcommand{\thetellerr}{(\roman{tellerr})}
\newenvironment{tabeleq}{\begin{list}%
{\rm  (\roman{tellerr})\hfill}{\usecounter{tellerr} \leftmargin=1.1cm
\labelwidth=1.1cm \labelsep=0cm \parsep=0cm}
                         }{\end{list}}

\newcounter{tellerrr}
\renewcommand{\thetellerrr}{(\Roman{tellerrr})}
\newenvironment{tabelR}{\begin{list}%
{\rm  (\Roman{tellerrr})\hfill}{\usecounter{tellerrr} \leftmargin=1.1cm
\labelwidth=1.1cm \labelsep=0cm \parsep=0cm}
                         }{\end{list}}

\newcounter{proofstep}
\newcommand{\nextstep}{\refstepcounter{proofstep}\vertspace \par 
          \noindent{\bf Step \theproofstep} \hspace{5pt}}
\newcommand{\firststep}{\setcounter{proofstep}{0}\nextstep}

\newcommand{\Ni}{\mathds{N}}
\newcommand{\Qi}{\mathds{Q}}
\newcommand{\Ri}{\mathds{R}}
\newcommand{\Ci}{\mathds{C}}
\newcommand{\Ti}{\mathds{T}}
\newcommand{\Zi}{\mathds{Z}}
\newcommand{\Fi}{\mathds{F}}
\newcommand{\Ki}{\mathds{K}}

\renewcommand{\proofname}{{\bf Proof}}

\newcommand{\vertspace}{\vskip10.0pt plus 4.0pt minus 6.0pt}

\newcommand{\simh}{{\stackrel{{\rm cap}}{\sim}}}
\newcommand{\ad}{{\mathop{\rm ad}}}
\newcommand{\Ad}{{\mathop{\rm Ad}}}
\newcommand{\alg}{{\mathop{\rm alg}}}
\newcommand{\clalg}{{\mathop{\overline{\rm alg}}}}
\newcommand{\Aut}{\mathop{\rm Aut}}
\newcommand{\arccot}{\mathop{\rm arccot}}
\newcommand{\capp}{{\mathop{\rm cap}}}
\newcommand{\rcapp}{{\mathop{\rm rcap}}}
\newcommand{\diam}{\mathop{\rm diam}}
\newcommand{\divv}{\mathop{\rm div}}
\newcommand{\dom}{\mathop{\rm dom}}
\newcommand{\codim}{\mathop{\rm codim}}
\newcommand{\RRe}{\mathop{\rm Re}}
\newcommand{\IIm}{\mathop{\rm Im}}
\newcommand{\tr}{{\mathop{\rm Tr \,}}}
\newcommand{\Tr}{{\mathop{\rm Tr \,}}}
\newcommand{\Vol}{{\mathop{\rm Vol}}}
\newcommand{\card}{{\mathop{\rm card}}}
\newcommand{\rank}{\mathop{\rm rank}}
\newcommand{\supp}{\mathop{\rm supp}}
\newcommand{\sgn}{\mathop{\rm sgn}}
\newcommand{\essinf}{\mathop{\rm ess\,inf}}
\newcommand{\esssup}{\mathop{\rm ess\,sup}}
\newcommand{\Int}{\mathop{\rm Int}}
\newcommand{\lcm}{\mathop{\rm lcm}}
\newcommand{\loc}{{\rm loc}}
\newcommand{\HS}{{\rm HS}}
\newcommand{\Trn}{{\rm Tr}}
\newcommand{\n}{{\rm N}}
\newcommand{\WOT}{{\rm WOT}}

\newcommand{\at}{@}

\newcommand{\mod}{\mathop{\rm mod}}
\newcommand{\spann}{\mathop{\rm span}}
\newcommand{\one}{\mathds{1}}

\hyphenation{groups}
\hyphenation{unitary}

\newcommand{\tfrac}[2]{{\textstyle \frac{#1}{#2}}}

\newcommand{\ca}{{\cal A}}
\newcommand{\cb}{{\cal B}}
\newcommand{\cc}{{\cal C}}
\newcommand{\cd}{{\cal D}}
\newcommand{\ce}{{\cal E}}
\newcommand{\cf}{{\cal F}}
\newcommand{\ch}{{\cal H}}
\newcommand{\chs}{{\cal HS}}
\newcommand{\ci}{{\cal I}}
\newcommand{\ck}{{\cal K}}
\newcommand{\cl}{{\cal L}}
\newcommand{\cm}{{\cal M}}
\newcommand{\cn}{{\cal N}}
\newcommand{\co}{{\cal O}}
\newcommand{\cp}{{\cal P}}
\newcommand{\cs}{{\cal S}}
\newcommand{\ct}{{\cal T}}
\newcommand{\cx}{{\cal X}}
\newcommand{\cy}{{\cal Y}}
\newcommand{\cz}{{\cal Z}}

\thispagestyle{empty}

\vspace*{1cm}
\begin{center}
{\Large\bf Operators with continuous kernels} \\[10mm]

\large W. Arendt$^1$ and A.F.M. ter Elst$^2$

\end{center}

\vspace{5mm}

\begin{center}
{\bf Abstract}
\end{center}

\begin{list}{}{\leftmargin=1.7cm \rightmargin=1.7cm \listparindent=10mm 
   \parsep=0pt}
\item
Let $\Omega \subset \Ri^d$ be open.
We investigate conditions under which an operator $T$ on $L_2(\Omega)$
has a continuous kernel $K \in C(\overline \Omega \times \overline \Omega)$.
In the centre of our interest is the condition 
$T L_2(\Omega) \subset C(\overline \Omega)$, which one 
knows for many semigroups generated by elliptic operators.
This condition implies that $T^3$ has a kernel in 
$C(\overline \Omega \times \overline \Omega)$ if $T$ is self-adjoint
and $\Omega$ is bounded, and the power $3$ is best possible.
We also analyse Mercer's theorem in our context.
\end{list}

\vspace{1cm}
\noindent
February 2019

\vspace{5mm}

\vspace{5mm}

\vspace{15mm}

\noindent
{\bf Home institutions:}    \\[3mm]
\begin{tabular}{@{}cl@{\hspace{10mm}}cl}
1. & Institute of Applied Analysis  & 
  2. & Department of Mathematics   \\
& University of Ulm   & 
  & University of Auckland   \\
& Helmholtzstr.\ 18 & 
  & Private Bag 92019  \\
& 89081 Ulm & 
  & Auckland 1142 \\
& Germany  & 
  & New Zealand  
\end{tabular}

\newpage

\section{Introduction} \label{Skernel1}

Kernel operators play an important role in analysis.
For example, the kernels of diffusion semigroups (heat kernels)
are of considerable interest to analyse the evolution (see Davies \cite{Dav2}
and Ouhabaz \cite{Ouh5}).
Let $\Omega \subset \Ri^d$ be an open bounded set and $T \in \cl(L_2(\Omega))$.
In many cases one is able to prove that $T$ has a measurable kernel
via the Dunford--Pettis criterion or by showing that $T$ is Hilbert--Schmidt.
But then it is frequently not easy to decide whether the kernel is continuous.
The results in the literature mainly establish stronger results such as 
H\"older continuity under quite strong hypotheses.
But just continuity is important.
For example, it is required for the trace formula in the context of Mercer's theorem.
A property which is frequently obtained automatically for semigroups
or resolvents, is that the operator maps $L_2(\Omega)$ into $C(\overline \Omega)$.
It is this property that we investigate in the present paper.
One of our main results is the following.

\begin{thm} \label{tkernel101}
Let $\Omega \subset \Ri^d$ be open and bounded.
Let $T$ be a self-adjoint bounded operator on $L_2(\Omega)$ such that 
$T L_2(\Omega) \subset C(\overline \Omega)$.
Then $T^3$ has a kernel in $C(\overline \Omega \times \overline \Omega)$.
\end{thm}

Of course, to say that $K \in C(\overline \Omega \times \overline \Omega)$
is a kernel of $T^3$ means that 
\[
(T^3 u)(x) 
= \int_\Omega K(x,y) \, u(y) \, dy
\]
for all $u \in L_2(\Omega)$ and $x \in \overline \Omega$.

We show by an example that $T^2$ does not need to have a kernel 
in $C(\Omega \times \Omega)$, 
even if $T$ is positive (in the sense of Hilbert spaces).
The optimal power $3$ demands some particular efforts.
In a previous paper \cite{AE8} we proved that 
$T^4$ has a kernel in $C(\overline \Omega \times \overline \Omega)$.
If $T$ is a positive operator on $L_2(\Omega)$ such that 
$T L_2(\Omega) \subset C(\overline \Omega)$, 
then we shall show that $T^{2+\varepsilon}$ 
has a kernel in $C(\overline \Omega \times \overline \Omega)$
for all $\varepsilon > 0$
(and this is optimal by what we said above).
Conversely, if a positive operator $T$ has a kernel in 
$C(\overline \Omega \times \overline \Omega)$, then 
Mercer's theorem shows that $T$ is trace class 
and we shall show that $T^{1/2} L_2(\Omega) \subset C(\overline \Omega)$
and that this result is optimal.

We also present results on unbounded domains.
Here our arguments give a nice result for a semigroup $S$ on $L_2(\Ri^d)$
which has Gaussian bounds.
If both $S_t L_2(\Ri^d) \subset C(\Ri^d)$ and 
$S_t^* L_2(\Ri^d) \subset C(\Ri^d)$ for all $t > 0$, then 
$S_t$ has a continuous kernel.
Finally we present a version of Mercer's theorem which is more 
general than the classical result and which fits well with our results.
In the last section examples are given.

\section{Continuous kernels, general $\Omega$} \label{Skernel2}

Let $\Omega \subset \Ri^d$ be open, non-empty and let $X$ be a set such that 
$\Omega \subset X \subset \overline \Omega$.
We provide $C(X)$ with the Fr\'echet topology
of uniform convergence on compact subsets of $X$.
Then an operator $T \colon L_2(\Omega) \to C(X)$ is 
compact if and only if for every sequence $(u_n)_{n \in \Ni}$ in $L_2(\Omega)$
such that $\lim_{n \to \infty} u_n = 0$ weakly in $L_2(\Omega)$ it follows 
that $\lim_{n \to \infty} \sup_{x \in F} |(T u_n)(x)| = 0$ for all non-empty
compact $F \subset X$.

\begin{prop} \label{pkernel201}
Let $\Omega \subset \Ri^d$ be open, non-empty and let $X$ be a set such that 
$\Omega \subset X \subset \overline \Omega$.
Let $T_1,T_2 \in \cl(L_2(\Omega))$.
Suppose that $T_1 L_2(\Omega) \subset C(X)$ and
$T_2 L_2(\Omega) \subset C(X)$.
Then the following are valid.
\begin{tabel}
\item \label{pkernel201-1}
There exists a measurable, separately continuous 
function $K \colon X \times X \to \Ci$
such that $K$ is bounded on compact subsets of $X \times X$,
the function $K(x,\cdot) \in L_2(\Omega)$ for all $x \in X$ and 
\[
(T_2 \, T_1^* u)(x)
= \int_\Omega K(x,y) \, u(y) \, dy
\]
for all $u \in L_2(\Omega)$ and $x \in X$.
\item \label{pkernel201-2}
If in addition $T_1 \colon L_2(\Omega) \to C(X)$ is compact
or $T_2 \colon L_2(\Omega) \to C(X)$ is compact,
then the kernel $K$ in Statement~\ref{pkernel201-1} is continuous.
\end{tabel}
\end{prop}
\begin{proof}
`\ref{pkernel201-1}'.
If $F \subset X$ is compact, then the operator $u \mapsto (T_1 u)|_F$ is 
bounded from $L_2(\Omega)$ into $C(F)$ by the closed graph theorem.
Hence it follows
from the Riesz representation theorem that for all $x \in X$ 
there exists a $k^{(1)}_x \in \cl_2(\Omega)$ such that 
\[
(T_1 u)(x)
= (u, k_x^{(1)})_{L_2(\Omega)}
\]
for all $u \in L_2(\Omega)$.
Then $x \mapsto k_x^{(1)}$ is bounded from compact subsets of
$X$ into $L_2(\Omega)$.
Clearly the map $x \mapsto k^{(1)}_x$ is continuous from $X$ 
into $(L_2(\Omega),w)$, the space $L_2(\Omega)$ provided with the weak topology.
We can define similarly the functions $k^{(2)}_x$
with respect to $T_2$.
Define $K \colon X \times X \to \Ci$ by
\begin{equation}
K(x,y)
= (k^{(1)}_y, k^{(2)}_x)_{L_2(\Omega)}
.  
\label{epkernel201;4}
\end{equation}
Then $K$ is bounded on compact subsets of $X \times X$ and separately continuous.
Hence $K$ is measurable by \cite{AliprantisBorder} Lemma~4.51.
If $u \in L_2(\Omega)$ and $x \in X$, then 
\begin{eqnarray*}
(T_2 \, T_1^* u)(x)
& = & (T_1^* u, k^{(2)}_x)_{L_2(\Omega)}
= (u, T_1 k^{(2)}_x)_{L_2(\Omega)}
= \int_\Omega u(y) \, \overline{ (T_1 k^{(2)}_x)(y) } \, dy  \\
& = & \int_\Omega K(x,y) \, u(y) \, dy
\end{eqnarray*}
since 
$(T_1 k^{(2)}_x)(y)
= (k^{(2)}_x, k^{(1)}_y)_{L_2(\Omega)} 
= \overline{K(x,y)}$ for all $y \in \Omega$.
Moreover, $K(x,\cdot) = \overline{T_1 k^{(2)}_x} \in L_2(\Omega)$.

`\ref{pkernel201-2}'.
Suppose that the operator $T_2 \colon L_2(\Omega) \to C(X)$ is compact.
(The proof for $T_1$ is similar.)
Let $x,x_1,x_2,\ldots,y,y_1,y_2,\ldots \in X$ and 
suppose that $\lim x_n = x$ and $\lim y_n = y$ in~$X$.
Let $F = \{ x, x_1,x_2,\ldots \} $.
Then $F$ is compact and $F \subset X$.
Now $\lim k^{(1)}_{y_n} = k^{(1)}_y$ weakly in $L_2(\Omega)$.
Hence by assumption
$\lim_{n \to \infty} T_2 k^{(1)}_{y_n} = T_2 k^{(1)}_y$
uniformly on $F$.
If $n \in \Ni$, then 
\begin{eqnarray*}
|K(x_n,y_n) - K(x,y)|
& = & |(k^{(1)}_{y_n}, k^{(2)}_{x_n})_{L_2(\Omega)} 
     - (k^{(1)}_y, k^{(2)}_x)_{L_2(\Omega)}|  \\
& = & |(T_2 k^{(1)}_{y_n})(x_n) - (T_2 k^{(1)}_y)(x)|  \\
& \leq & |(T_2 k^{(1)}_{y_n})(x_n) - (T_2 k^{(1)}_y)(x_n)| 
   + |(T_2 k^{(1)}_y)(x_n) - (T_2 k^{(1)}_y)(x)|  
\end{eqnarray*}
for all $n \in \Ni$ and the continuity of $K$ follows.
\end{proof}

A special case of Proposition~\ref{pkernel201}\ref{pkernel201-1} has been proved by 
\cite{KLVW} Proposition~3.3, where positivity improving self-adjoint semigroups
given by kernels are investigated.

For completeness we mention the following uniqueness for separately
continuous functions.

\begin{lemma} \label{lkernel240}
Let $\Omega \subset \Ri^d$ be open, non-empty and let $X$ be a set such that 
$\Omega \subset X \subset \overline \Omega$.
Let $K \colon X \times X \to \Ci$ be separately continuous
and suppose that $K = 0$ almost everywhere on $\Omega \times \Omega$.
Then $K = 0$ pointwise on $X \times X$.
\end{lemma}
\begin{proof}
This follows from Fubini's theorem.
\end{proof}

Hence a separately continuous kernel is unique if it exists.
This means: let $\Omega \subset \Ri^d$ be open, non-empty and let $X$ be a set such that 
$\Omega \subset X \subset \overline \Omega$, let 
$K \colon X \times X \to \Ci$ be separately continuous with $K(x,\cdot) \in L_2(\Omega)$
and $\int_\Omega K(x,y) \, u(y) \, dy = 0$ for almost every $x \in \Omega$ and 
$u \in L_2(\Omega)$, then $K = 0$ pointwise on $X \times X$.

\medskip

Without the additional compactness condition the joint continuity
fails in general.
We next give an example of a bounded set $\Omega$ and a positive operator
$T$ on $L_2(\Omega)$ such that $T L_2(\Omega) \subset C(\overline \Omega)$,
but the kernel of the operator $T^2 = T T^*$ is not (jointly) continuous,
even not on $\Omega \times \Omega$.

\begin{exam} \label{xkernel203}
Choose $\Omega = (-1,1)$.
We first construct an operator $T \in \cl(L_2(\Omega))$ such that 
$T L_2(\Omega) \subset C(\overline \Omega)$ and 
$T^* L_2(\Omega) \subset C(\overline \Omega)$, but 
the kernel of the operator $T T^*$ is not (jointly) continuous, 
since it is not continuous at $(0,0)$.
We then construct a self-adjoint counter-example and finally a 
positive (self-adjoint) counter-example.

\firststep
Fix $\tau \in C_c^\infty( (-1,1) \times (-1,1) )$
such that $0 \leq \tau \leq \one$ and 
$\tau_{[-\frac{1}{2}, \frac{1}{2}] \times [-\frac{1}{2}, \frac{1}{2}]} = \one$.
Define $K \colon [-1,1] \times [-1,1] \to \Ri$ by 
\[
K(x,z) = \sum_{n=1}^\infty 3^n \, \tau(10^n (x - 2^n), 9^n (z - 3^{-n}) )
.  \]
Then $K$ is continuous on $([-1,1] \times [-1,1]) \setminus \{ (0,0) \} $ and 
\[
\supp K
\subset \bigcup_{n=1}^\infty
  \Big(  [2^{-n} - 10^{-n} , 2^{-n} + 10^{-n}] \times [3^{-n} - 9^{-n}, 3^{-n} + 9^{-n}]
  \Big)
.  \]
Moreover,  
\[
\|K\|_{L_2(\Omega \times \Omega)}^2
\leq \sum_{n=1}^\infty 9^n \cdot 2 \cdot 10^{-n} \cdot 2 \cdot 9^{-n}
= 4 \sum_{n=1}^\infty 10^{-n}
< \infty
.  \]
Hence one can define the Hilbert--Schmidt operator $T \colon L_2(\Omega) \to L_2(\Omega)$
by 
\[
(T u)(x) 
= \int_\Omega K(x,z) \, u(z) \, dz
.  \]
We choose $T_1 = T_2 = T$ in Proposition~\ref{pkernel201}.

Define $K^{(2)} \colon \overline \Omega \times \overline \Omega \to \Ri$ by 
\[
K^{(2)}(x,y) 
= \int_\Omega K(x,z) \, K(y,z) \, dz
.  \]
Then $K^{(2)}$ is continuous on 
$(\overline \Omega \times \overline \Omega) \setminus \{ (0,0) \} $.
Moreover,
$(T \, T^* u)(x) = \int_\Omega K^{(2)}(x,y) \, u(y) \, dy$
for all $u \in L_2(\Omega)$ and almost every $x \in \Omega$.
Note that 
\[
K^{(2)}(2^{-n}, 2^{-n}) 
\geq \int_{[3^{-n} - 9^{-n}, 3^{-n} + 9^{-n}]} 
        (3^n)^2 \, \Big( \tau(0, 9^n (z - 3^{-n})) \Big)^2
\geq 1
\]
for all $n \in \Ni$ and $K^{(2)}(0,0) = 0 = \lim_{n \to \infty} K^{(2)}(- 2^{-n}, - 2^{-n})$.
So $K^{(2)}$ is not continuous at $(0,0)$.

Let $u \in L_2(\Omega)$ and $y \in \overline \Omega$.
We shall show that $T u$ is continuous at~$y$.
This is trivial if $y < 0$ and it easily follows from the Lebesgue dominated 
convergence theorem if $y > 0$.
So it remains to show continuity of $T u$ at $0$.
Let $x \in \Omega \setminus \{ 0 \} $.
There is at most one $n \in \Ni$ such that 
$x \in [2^{-n} - 10^{-n} , 2^{-n} + 10^{-n}]$.
Then 
\[
|(T u)(x)| 
\leq \sqrt{2} \cdot 3^n \int_{[3^{-n} - 9^{-n}, 3^{-n} + 9^{-n}]} |u(z)| \, dz
= (f_n , |u|)_{L_2(\Omega)}
,  \]
where $f_n = \sqrt{2} \cdot 3^n \, \one_{[3^{-n} - 9^{-n}, 3^{-n} + 9^{-n}]}$.
Since the family $(f_k)_{k \in \Ni}$ is orthonormal, it follows that 
$\lim_{k \to \infty} (f_k , |u|)_{L_2(\Omega)} = 0$.
Hence $\lim_{x \to 0} (T u)(x) = 0$.
We proved that $T L_2(\Omega) \subset C(\overline \Omega)$.

Finally we show that $T^* L_2(\Omega) \subset C(\overline \Omega)$.
Let $u \in L_2(\Omega)$.
Again it is easy to show continuity on $\overline \Omega \setminus \{ 0 \} $,
so we have to show continuity at~$0$.
Let $x \in \Omega \setminus \{ 0 \} $.
There is at most one $n \in \Ni$ such that 
$x \in [3^{-n} - 9^{-n}, 3^{-n} + 9^{-n}]$.
Then 
\[
|(T^* u)(x)|
\leq \int_\Omega K(z,x) \, |u(z)| \, dz
\leq 3^n \int_{[2^{-n} - 10^{-n} , 2^{-n} + 10^{-n}]} |u(z)| \, dz
\leq 3^n \, \sqrt{2 \cdot 10^{-n}} \, \|u\|_2
.  \]
So $\lim_{x \to 0} (T^* u)(x) = 0$.
Hence $T^* L_2(\Omega) \subset C(\overline \Omega)$.

\nextstep
Define $\widehat T = T + T^*$.
Then $\widehat T$ is self-adjoint and $\widehat T L_2(\Omega) \subset C(\overline \Omega)$.
Define $\widetilde K \colon \overline \Omega \times \overline \Omega \to \Ri$ by 
\[
\widetilde K(x,y) 
= \int_\Omega \Big(K(x,z) + K(z,x)\Big) \, \Big(K(y,z) + K(z,y)\Big) \, dz
.  \]
Then $(\widehat T \, \widehat T^* u)(x) = \int_\Omega \widetilde K(x,y) \, u(y) \, dy$
for all $u \in L_2(\Omega)$ and almost every $x \in \Omega$.
As before $\widetilde K(- 2^{-n}, - 2^{-n}) = 0 = \widetilde K(0,0)$ for all $n \in \Ni$.
Also $\widetilde K \geq K^{(2)}$.
So $\widetilde K(2^{-n}, 2^{-n}) \geq K^{(2)}(2^{-n}, 2^{-n}) \geq 1$
for all $n \in \Ni$ and $\widetilde K$ is not continuous.

\nextstep
Define $\widetilde T = |\widehat T|$.
Then $\widetilde T$ is positive and $\widetilde T^2 = \widehat T^2$ does 
not have a continuous kernel on $\overline \Omega \times \overline \Omega$.
Since $\widehat T$ is self-adjoint, there exists a unitary operator 
$U$ such that $|\widehat T| = \widehat T \circ U$.
Then 
$\widetilde T L_2(\Omega) 
= \widehat T (U L_2(\Omega)) 
= \widehat T L_2(\Omega)
\subset C(\overline \Omega)$ 
as required.
\end{exam}

For three operators and a compactness condition we next deduce
joint continuity of the kernel on $\overline \Omega \times \overline \Omega$,
even if $\Omega$ is unbounded.

\begin{thm} \label{tkernel202}
Let $\Omega \subset \Ri^d$ be open, non-empty and let $X$ be a set such that 
$\Omega \subset X \subset \overline \Omega$.
Let $T_1,T_2,T_3 \in \cl(L_2(\Omega))$.
Suppose that $T_1 L_2(\Omega) \subset C(X)$
and $T_3 L_2(\Omega) \subset C(X)$.
Moreover, suppose that $T_2$ is a compact operator from $L_2(\Omega)$ into $L_2(\Omega)$.
Then there exists a kernel $K \in C(X \times X)$ 
such that $K(x,\cdot) \in L_2(\Omega)$ for all $x \in X$ and
\[
(T_3 \, T_2 \, T_1^* u)(x)
= \int_\Omega K(x,y) \, u(y) \, dy
\]
for all $u \in L_2(\Omega)$ and $x \in X$.
\end{thm}
\begin{proof}
Note that $T_3 \, T_2 \, T_1^* = T_3 \, (T_1 \, T_2^*)^*$.
Since $T_2^* \colon L_2(\Omega) \to L_2(\Omega)$ is compact, the 
operator $T_1 \, T_2^* \colon L_2(\Omega) \to C(X)$ is compact.
Now it follows from Proposition~\ref{pkernel201}\ref{pkernel201-2}
that $T_3 \, (T_1 \, T_2^*)^*$ has a kernel in 
$C(X \times X)$.
\end{proof}

Finally we present an application for semigroups.
Note that by Proposition~\ref{pkernel201} the hypotheses in the next result
imply that each semigroup operator $S_t$ has a separately continuous kernel.
Under the additional hypothesis of Gaussian bounds we show that this 
kernel is jointly continuous.
Second-order elliptic operators under diverse boundary conditions are known
to generate semigroups with Gaussian bounds
(see \cite{AE1}, \cite{Daners} and \cite{Ouh5}).

\begin{prop} \label{pkernel321}
Let $\Omega \subset \Ri^d$ be open, non-empty and let $X$ be a set such that 
$\Omega \subset X \subset \overline \Omega$.
Let $S$ be a semigroup in $L_2(\Omega)$
such that $S_t L_2(\Omega) \subset C(X)$ and 
$S_t^* L_2(\Omega) \subset C(X)$ for all $t > 0$.
Suppose the semigroup satisfies Gaussian bounds, that is 
there are $b,c,\omega > 0$ such that the separately continuous 
kernel $K_t \colon X \times X \to \Ci$ of $S_t$ 
satisfies
\[
|K_t(x,y)|
\leq c \, t^{-d/2} \, e^{-b |x-y|^2 t^{-1}} \, e^{\omega t}
\]
for all $x,y \in \Omega$ and $t > 0$.
Then $K_t$ is continuous for all $t > 0$.
\end{prop}
\begin{proof}
Since $S_{2t} = S_t \, (S_t^*)^*$ it follows from 
Proposition~\ref{pkernel201}\ref{pkernel201-1} that the operator 
$S_{2t}$ has a separately continuous kernel on 
$\overline \Omega \times \overline \Omega$ for all $t > 0$.
So we may assume that $K_t$ is separately continuous for all $t > 0$.
The semigroup property gives
\begin{equation}
K_{2t}(x,y)
= \int_\Omega K_t(x,z) \, K_t(z,y) \, dy
\label{epkernel321;1}
\end{equation}
for all $x,y \in \Omega$.
Then the Gaussian bounds together with the Lebesgue dominated convergence theorem
first give that (\ref{epkernel321;1}) extends to all $x,y \in X$
and then give the continuity of~$K_{2t}$.
\end{proof}

\section{Continuous kernels, bounded $\Omega$} \label{Skernel3}

Let $\Omega \subset \Ri^d$ be open and bounded.
Then one can easily characterise the operators $T \in \cl(L_2(\Omega))$
which map $L_2(\Omega)$ into $C(\overline \Omega)$.
Note that if $T L_2(\Omega) \subset C(\overline \Omega)$, then the operator 
$T \colon L_2(\Omega) \to C(\overline \Omega)$ is bounded by the closed 
graph theorem.

\begin{prop} \label{pkernel231}
Let $\Omega \subset \Ri^d$ be open and bounded.
Let $T \in \cl(L_2(\Omega))$.
Then the following are equivalent.
\begin{tabeleq} 
\item \label{pkernel231-1}
$T L_2(\Omega) \subset C(\overline \Omega)$.
\item \label{pkernel231-2}
There exists a continuous $k \colon \overline \Omega \to (L_2(\Omega),w)$
such that
\[
(T u)(x) = (f, k(x))_{L_2(\Omega)}
\]
for all $u \in L_2(\Omega)$ and $x \in \overline \Omega$.
\end{tabeleq}
If both conditions are valid, then 
$\|T\|_{L_2(\Omega) \to C(\overline \Omega)}
= \sup_{x \in \overline \Omega} \|k_x\|_{L_2(\Omega)}$.
\end{prop}

We leave the easy proof to the reader.
If $k \colon \overline \Omega \to (L_2(\Omega),w)$ is as in Condition~\ref{pkernel231-2},
then we frequently write $k_x = k(x)$ for all $x \in \overline \Omega$.
For convenience of the reader we include the following.

\begin{cor} \label{ckernel232}
Let $\Omega \subset \Ri^d$ be open and bounded. 
Let $T \in \cl(L_2(\Omega))$ and suppose that $T L_2(\Omega) \subset C(\overline \Omega)$.
Then $T$ is Hilbert--Schmidt and in particular 
$T$ is compact from $L_2(\Omega)$ into $L_2(\Omega)$.
\end{cor}
\begin{proof}
Let $k \colon \overline \Omega \to (L_2(\Omega),w)$ be as in Condition~\ref{pkernel231-2}
of Proposition~\ref{pkernel231}.
Let $(e_n)_{n \in \Ni}$ be an orthonormal basis for $L_2(\Omega)$.
Then 
\begin{eqnarray*}
\sum_{n=1}^\infty \|T e_n\|_2^2
= \sum_{n=1}^\infty \int_\Omega |(T e_n)(x)|^2 \, dx
& = & \sum_{n=1}^\infty \int_\Omega |(e_n, k_x)_{L_2(\Omega)}|^2 \, dx  \\
& = & \int_\Omega \sum_{n=1}^\infty |(e_n, k_x)_{L_2(\Omega)}|^2 \, dx
= \int_\Omega \|k_x\|_{L_2(\Omega)}^2 \, dx
< \infty
.  
\end{eqnarray*}
Hence $T$ is Hilbert--Schmidt and consequently compact.
\end{proof}

In general the operator in Corollary~\ref{ckernel232} is not trace class, 
see Example~\ref{xkernel404} below.
Also in general the operator in Corollary~\ref{ckernel232}
is not compact from $L_2(\Omega)$ into $C(\overline \Omega)$.
This is a stronger property that we descibe now.

\begin{cor} \label{ckernel233}
Let $\Omega \subset \Ri^d$ be open and bounded. 
Let $T \in \cl(L_2(\Omega))$ and suppose that $T L_2(\Omega) \subset C(\overline \Omega)$.
Let $k \colon \overline \Omega \to (L_2(\Omega),w)$ be as in Proposition~\ref{pkernel231}.
Then the following are equivalent.
\begin{tabeleq}
\item \label{ckernel233-1}
The operator $T$ is compact from $L_2(\Omega)$ into $C(\overline \Omega)$.
\item \label{ckernel233-2}
The map $x \mapsto \|k_x\|_{L_2(\Omega)}$ from 
$\overline \Omega$ into $\Ri$ is continuous.
\item \label{ckernel233-3}
The map $k \colon \overline \Omega \to L_2(\Omega)$ is continuous.
\item \label{ckernel233-4}
$\displaystyle 
 \lim_{N \to \infty} \: \sup_{x \in \overline \Omega} \: \sum_{n=N}^\infty |(e_n, k_x)|^2 = 0$.
\end{tabeleq}
\end{cor}
\begin{proof}
`\ref{ckernel233-1}$\Rightarrow$\ref{ckernel233-2}'.
Proposition~\ref{pkernel231}\ref{pkernel231-2} gives that the 
map $x \mapsto k_x$ is continuous from 
$\overline \Omega$ into $(L_2(\Omega),w)$.
Since $T$ is compact from $L_2(\Omega)$ into $C(\overline \Omega)$, 
it follows that the map 
$x \mapsto T k_x$ is continuous from 
$\overline \Omega$ into $C(\overline \Omega)$.
Hence the map $x \mapsto (T k_x)(x)$ is continuous from 
$\overline \Omega$ into $\Ci$.
Because $(T k_x)(x) = (k_x, k_x)_{L_2(\Omega)} = \|k_x\|_{L_2(\Omega)}^2$
for all $x \in \overline \Omega$,
the implication follows.

`\ref{ckernel233-2}$\Rightarrow$\ref{ckernel233-3}'.
Let $x,x_1,x_2,\ldots \in \overline \Omega$ and suppose that 
$\lim_{n \to \infty} x_n = x$ in $\overline \Omega$.
Then $\lim_{n \to \infty} k_{x_n} = k_x$ in $(L_2(\Omega),w)$
by Proposition~\ref{pkernel231}\ref{pkernel231-2}.
Since $\lim_{n \to \infty} \|k_{x_n}\|_{L_2(\Omega)} = \|k_x\|_{L_2(\Omega)}$
by assumption, one deduces that
$\lim_{n \to \infty} k_{x_n} = k_x$ in $L_2(\Omega)$.

`\ref{ckernel233-3}$\Rightarrow$\ref{ckernel233-4}'.
For all $N \in \Ni$ define $Q_N \colon L_2(\Omega) \to L_2(\Omega)$ by 
$Q_N u = \sum_{n=N}^\infty (u,e_n)_{L_2(\Omega)} e_n$.
Then $\lim_{N \to \infty} Q_N u = 0$ in $L_2(\Omega)$ for all $u \in L_2(\Omega)$.
Hence if $F$ is a compact subset of $L_2(\Omega)$, then 
$\lim_{N \to \infty} \sup_{u \in F} \|Q_N u\|_{L_2(\Omega)} = 0$.
By assumption the map $k \colon \overline \Omega \to L_2(\Omega)$ is continuous.
Since $\overline \Omega$ is compact, the set 
$F = \{ k_x : x \in \overline \Omega \} $ is compact in $L_2(\Omega)$.
So $\lim_{N \to \infty} \sup_{x \in \overline \Omega} \|Q_N k_x\|_{L_2(\Omega)} = 0$.
This is Condition~\ref{ckernel233-4}.

`\ref{ckernel233-4}$\Rightarrow$\ref{ckernel233-1}'.
For all $N \in \Ni$ define $T_N \colon L_2(\Omega) \to C(\overline \Omega)$ by 
$(T_N u)(x) = \sum_{n=1}^{N-1} (u, e_n) \, (e_n, k_x)$.
Then $T_N$ has finite rank, hence it is compact.
Let $N \in \Ni$ and $u \in L_2(\Omega)$.
Then 
\begin{eqnarray*}
|((T - T_N) u)(x)|
& = & |(u,k_x) - (T_N u)(x)|  \\
& = & \Big| \sum_{n=N}^\infty (u, e_n) \, (e_n, k_x) \Big|  \\
& \leq & \Big( \sum_{n=N}^\infty |(u, e_n)|^2 \Big)^{1/2} 
         \Big( \sum_{n=N}^\infty |(e_n, k_x)|^2 \Big)^{1/2}  
\leq \|u\|_{L_2(\Omega)}
         \Big( \sum_{n=N}^\infty |(e_n, k_x)|^2 \Big)^{1/2} 
\end{eqnarray*}
for all $x \in \overline \Omega$.
So 
\[
\|T - T_N\|_{L_2(\Omega) \to C(\overline \Omega)}
\leq \sup_{x \in \overline \Omega} \: \Big( \sum_{n=N}^\infty |(e_n, k_x)|^2 \Big)^{1/2} 
\]
and $\lim_{N \to \infty} T_N = T$ in $\cl(L_2(\Omega) , C(\overline \Omega))$.
\end{proof}

In view of Corollary~\ref{ckernel232},
Theorem~\ref{tkernel202} takes a very simple form if $\Omega$ is bounded.

\begin{cor} \label{ckernel204.4}
Let $\Omega \subset \Ri^d$ be open and bounded.
Let $T_1,T_2,T_3 \in \cl(L_2(\Omega))$.
Suppose that $T_k L_2(\Omega) \subset C(\overline \Omega)$ for all $k \in \{ 1,2,3 \} $.
Then there exists a  $K \in C(\overline \Omega \times \overline \Omega)$
such that 
\[
(T_3 \, T_2 \, T_1^* u)(x)
= \int_\Omega K(x,y) \, u(y) \, dy
\]
for all $u \in L_2(\Omega)$ and $x \in \overline \Omega$.
\end{cor}

The following theorem is in the spirit of Mercer's theorem.

\begin{thm} \label{tkernel206}
Let $\Omega \subset \Ri^d$ be open and bounded.
Let $T_1,T_2 \in \cl(L_2(\Omega))$.
Suppose that $T_1 L_2(\Omega) \subset C(\overline \Omega)$ and
$T_2 L_2(\Omega) \subset C(\overline \Omega)$.
Moreover, suppose in addition that 
$T_1 \colon L_2(\Omega) \to C(\overline \Omega)$ is compact
or $T_2 \colon L_2(\Omega) \to C(\overline \Omega)$ is compact.
Let $K \in C(\overline \Omega \times \overline \Omega)$ 
be the kernel of the operator $T_2 \, T_1^*$.
Let $(e_n)_{n \in \Ni}$ be an orthonormal basis for $L_2(\Omega)$.
For all $n \in \Ni$ define $u_n = T_2 e_n$ and $v_n = T_1 e_n$.
Note that $u_n,v_n \in C(\overline \Omega)$.
Then 
\[
K = \sum_{n=1}^\infty u_n \otimes \overline{v_n}
\]
and the series converges in $C(\overline \Omega \times \overline \Omega)$.
\end{thm}
\begin{proof}
We use the notation as in the proof of Proposition~\ref{pkernel201}.
If $x,y \in \overline \Omega$, then 
\begin{eqnarray*}
K(x,y)
& = & (k^{(1)}_y, k^{(2)}_x)_{L_2(\Omega)}  \\
& = & \sum_{n=1}^\infty (k^{(1)}_y, e_n)_{L_2(\Omega)} \, (e_n, k^{(2)}_x)_{L_2(\Omega)}
= \sum_{n=1}^\infty \overline{ (T_1 e_n)(y) } \, (T_2 e_n)(x)
.
\end{eqnarray*}
So it remains to show the convergence in $C(\overline \Omega \times \overline \Omega)$.

Suppose that $T_2 \colon L_2(\Omega) \to C(\overline \Omega)$ is compact.
(The proof is similar in the other case.)
Let $N \in \Ni$ and let $x,y \in \overline \Omega$.
Then 
\begin{eqnarray*}
\sum_{n=N}^\infty |(u_n \otimes \overline{v_n})(x,y)|
& = & \sum_{n=N}^\infty 
   |(k^{(1)}_y, e_n)_{L_2(\Omega)}| \, |(e_n, k^{(2)}_x)_{L_2(\Omega)}|  \\
& \leq & \Big( \sum_{n=N}^\infty |(k^{(1)}_y, e_n)_{L_2(\Omega)}|^2 \Big)^{1/2}
         \Big( \sum_{n=N}^\infty |(e_n, k^{(2)}_x)_{L_2(\Omega)}|^2 \Big)^{1/2}  \\
& \leq & \|T_1\|_{L_2(\Omega) \to C(\overline \Omega)} \, 
     \Big( \sum_{n=N}^\infty |(e_n, k^{(2)}_x)_{L_2(\Omega)}|^2 \Big)^{1/2}
, 
\end{eqnarray*}
where we used the end of Proposition~\ref{pkernel231} in the last step.
Hence 
\[
\lim_{N \to \infty} \Big\| \sum_{n=N}^\infty |u_n \otimes \overline{v_n}| \Big\|_{C(\overline \Omega \times \overline \Omega)}
= 0
\]
by Corollary~\ref{ckernel233}\ref{ckernel233-1}$\Rightarrow$\ref{ckernel233-4}
and the result follows.
\end{proof}

Under the same conditions a trace formula is valid.

\begin{thm} \label{tkernel205}
Let $\Omega \subset \Ri^d$ be open and bounded.
Let $T_1,T_2 \in \cl(L_2(\Omega))$.
Suppose that $T_1 L_2(\Omega) \subset C(\overline \Omega)$ and
$T_2 L_2(\Omega) \subset C(\overline \Omega)$.
Moreover, suppose in addition that 
$T_1 \colon L_2(\Omega) \to C(\overline \Omega)$ is compact
or $T_2 \colon L_2(\Omega) \to C(\overline \Omega)$ is compact.
Let $K \in C(\overline \Omega \times \overline \Omega)$ 
be the kernel of the operator $T_2 \, T_1^*$.
Then $T_2 \, T_1^*$ is trace class and 
\[
\Tr(T_2 \, T_1^*)
= \int_\Omega K(x,x) \, dx
.  \]
\end{thm}
\begin{proof}
Clearly $T_2 \, T_1^*$ is trace class since it is the product
of two Hilbert--Schmidt operators.
Let $(e_n)_{n \in \Ni}$ be an orthonormal basis for $L_2(\Omega)$.
Then Theorem~\ref{tkernel206} gives
\begin{eqnarray*}
\Tr(T_2 \, T_1^*)
& = & \Tr(T_1^* \, T_2)
= (T_2, T_1)_\HS
= \sum_{n=1}^\infty (T_2 e_n, T_1 e_n)_{L_2(\Omega)}  \\
& = & \sum_{n=1}^\infty \int_\Omega (T_2 e_n)(x) \, \overline{ (T_1 e_n)(x)} \, dx 
= \int_\Omega \sum_{n=1}^\infty (T_2 e_n)(x) \, \overline{ (T_1 e_n)(x)} \, dx  
= \int_\Omega K(x,x) \, dx
\end{eqnarray*}
as required.
\end{proof}

We next give an example of a bounded set $\Omega$ and a positive (self-adjoint)
operator $T$ which maps $L_2(\Omega)$ into $C(\overline \Omega)$ 
such that the kernel of $T$ is not bounded.

\begin{exam} \label{xkernel207}
Choose $\Omega = (-1,1)$ and for all $n \in \Ni_0$ let $P_n$ be the 
$n$-th Legendre polynomial.
For all $n \in \Ni_0$ define $e_n = \sqrt{\frac{2n+1}{2}} \, P_n$.
Then $(e_n)_{n \in \Ni_0}$ is an orthonormal basis for $L_2(\Omega)$.
Define $T \in \cl(L_2(\Omega))$ by 
\[
T u = \sum_{n=1}^\infty \frac{1}{n^2} \, (u,e_n)_{L_2(\Omega)} \, e_n
.  \]
Clearly $T$ is positive.
Let $u \in L_2(\Omega)$.
Then 
\[
T u = \sum_{n=1}^\infty \frac{1}{n^2} \, \sqrt{\frac{2n+1}{2}} \, (u,e_n)_{L_2(\Omega)} \, P_n
.  \]
Since $\Big( \frac{1}{n^2} \, \sqrt{\frac{2n+1}{2}} \Big)_{n \in \Ni_0} \in \ell_2(\Ni_0)$
it follows that 
$\Big( \frac{1}{n^2} \, \sqrt{\frac{2n+1}{2}} \, (u,e_n)_{L_2(\Omega)} \Big)_{n \in \Ni_0} \in \ell_1(\Ni_0)$.
Moreover $\|P_n\|_{C(\overline \Omega)} = P_n(1) = 1$ for all $n \in \Ni_0$.
Therefore 
\[
\sum_{n=1}^\infty \frac{1}{n^2} \, \sqrt{\frac{2n+1}{2}} \, (u,e_n)_{L_2(\Omega)} \, P_n
\in C(\overline \Omega)
\]
and $T u \in C(\overline \Omega)$.

Define $K \colon \overline \Omega \times \Omega \to \Ci$ by 
\[
K(x,y) = \sum_{n=1}^\infty \frac{1}{n^2} \, \frac{2n+1}{2} \, P_n(x) \, P_n(y)
.  \]
Note that the series converges by \cite{Szego} Theorem~8.21.2.
Then $(T u)(x) = \int_\Omega K(x,y) \, u(y) \, dy$ for all $u \in L_2(\Omega)$
and $x \in \overline \Omega$.

Finally,
\begin{eqnarray*}
\sup_{x \in (0,1)} K(x,x)
& = & \sup_{x \in (0,1)} \sup_{N \in \Ni} 
   \sum_{n=1}^N \frac{1}{n^2} \, \frac{2n+1}{2} \, |P_n(x)|^2  \\
& = & \sup_{N \in \Ni} \sup_{x \in (0,1)} 
   \sum_{n=1}^N \frac{1}{n^2} \, \frac{2n+1}{2} \, |P_n(x)|^2  \\
& = & \sup_{N \in \Ni} 
   \sum_{n=1}^N \frac{1}{n^2} \, \frac{2n+1}{2} \, |P_n(1)|^2  \\
& = & \sum_{n=1}^\infty \frac{1}{n^2} \, \frac{2n+1}{2} 
= \infty
.
\end{eqnarray*}
Hence the kernel of $T$ is not bounded.
\end{exam}

We next derive a kind of converse of Proposition~\ref{pkernel201}\ref{pkernel201-2} 
for self-adjoint operators.

\begin{prop} \label{pkernel308}
Let $\Omega \subset \Ri^d$ be open and bounded.
Let $T \in \cl(L_2(\Omega))$ be a self-adjoint operator.
Suppose there exists a $K \in C(\overline \Omega \times \overline \Omega)$ 
such that $(T^2 u)(x) = \int_\Omega K(x,y) \, u(y) \, dy$ for all $u \in L_2(\Omega)$
and almost every $x \in \Omega$.
Then $T L_2(\Omega) \subset C(\overline \Omega)$.
Moreover, the operator $T \colon L_2(\Omega) \to C(\overline \Omega)$ is compact.
\end{prop}
\begin{proof}
Since $K$ is continuous it follows that $T^2 L_2(\Omega) \subset C(\overline \Omega)$.
Hence $T^2$ is compact and consequently $T$ is compact.
There exists an orthonormal basis $(e_n)_{n \in \Ni}$ for $L_2(\Omega)$ and
$\lambda_1,\lambda_2,\ldots \in \Ri$ such that $T e_n = \lambda_n \, e_n$
for all $n \in \Ni$.
We assume that $\lambda_n \neq 0$ for all $n \in \Ni$.
(The other case is similar.)
Then $\lambda_n^2 \, e_n = T^2 e_n \in C(\overline \Omega)$ and $e_n \in C(\overline \Omega)$
for all $n \in \Ni$.
It follows from Mercer's theorem (see Theorem~\ref{tkernel502})
that the series $\sum \lambda_n^2 \, |e_n|^2$
converges uniformly on $\overline \Omega$.

Let $u \in L_2(\Omega)$.
Then $T u = \sum_{n=1}^\infty \lambda_n \, (u,e_n) \, e_n$ in $L_2(\Omega)$.
Now 
\[
\sum_{n=N}^\infty | \lambda_n \, (u,e_n) \, e_n |
\leq \Big( \sum_{n=N}^\infty |(u,e_n)|^2 \Big)^{1/2} 
     \Big( \sum_{n=N}^\infty | \lambda_n \, e_n |^2 \Big)^{1/2} 
\]
for all $N \in \Ni$.
Since $\sum_{n=1}^\infty | \lambda_n \, e_n |^2$ is bounded, it follows that 
$\sum \lambda_n \, (u,e_n) \, e_n$ converges in $C(\overline \Omega)$.
So $T u \in C(\overline \Omega)$.

Finally we prove compactness.
Let $(u_m)_{m \in \Ni}$ be a sequence in $L_2(\Omega)$ which converges
weakly to zero.
We shall show that $\lim_{m \to \infty} T u_m = 0$ in $C(\overline \Omega)$.
Let $\varepsilon > 0$.
Since $\sum \lambda_n^2 \, |e_n|^2$
converges uniformly on $\overline \Omega$, there exists an $N \in \Ni$ 
such that 
$\sum_{n=N}^\infty \lambda_n^2 \, |e_n(x)|^2 \leq \varepsilon^2$ for all 
$x \in \overline \Omega$.
There exists an $M \in \Ni$ such that 
$|(u_m,e_n)| \, \|\lambda_n \, e_n\|_\infty \leq \frac{\varepsilon}{N}$
for all $n \in \{ 1,\ldots,N \} $ and $m \in \Ni$ with $m \geq M$.
Let $m \in \Ni$ with $m \geq M$.
Then 
\begin{eqnarray*}
|(T u_m)(x)|
& \leq & \sum_{n=1}^N |\lambda_n \, (u_m,e_n) \, e_n(x)|
   + \sum_{n=N+1}^\infty |\lambda_n \, (u_m,e_n) \, e_n(x)|  \\
& \leq & \sum_{n=1}^N \frac{\varepsilon}{N}
   + \Big( \sum_{n=N+1}^\infty |(u_m,e_n)|^2 \Big)^{1/2} 
     \Big( \sum_{n=N+1}^\infty | \lambda_n \, e_n(x) |^2 \Big)^{1/2}   \\
& \leq & \varepsilon + \|u_m\|_{L_2(\Omega)} \, \varepsilon
\end{eqnarray*}
for all $x \in \overline \Omega$.
Since $(u_m)_{m \in \Ni}$ is bounded in $L_2(\Omega)$ one deduces 
that $\lim_{m \to \infty} T u_m = 0$ in $C(\overline \Omega)$.
\end{proof}

\begin{thm} \label{tkernel310}
Let $\Omega \subset \Ri^d$ be open and bounded.
Let $T \in \cl(L_2(\Omega))$ be a self-adjoint operator.
Then the following are equivalent.
\begin{tabeleq}
\item \label{tkernel310-1}
There exists a $K \in C(\overline \Omega \times \overline \Omega)$ 
such that $(T^2 u)(x) = \int_\Omega K(x,y) \, u(y) \, dy$ for all $u \in L_2(\Omega)$
and almost every $x \in \Omega$.
\item \label{tkernel310-2}
$T L_2(\Omega) \subset C(\overline \Omega)$ and 
the operator $T \colon L_2(\Omega) \to C(\overline \Omega)$ is compact.
\end{tabeleq}
\end{thm}
\begin{proof}
`\ref{tkernel310-1}$\Rightarrow$\ref{tkernel310-2}'.
This is Proposition~\ref{pkernel308}.

`\ref{tkernel310-2}$\Rightarrow$\ref{tkernel310-1}'.
This is a special case of Proposition~\ref{pkernel201}\ref{pkernel201-2}.
\end{proof}

\section{Positive operators} \label{Skernel4}

For positive operators one can improve Corollary~\ref{ckernel204.4}.

\begin{thm} \label{tkernel204.5}
Let $\Omega \subset \Ri^d$ be open and bounded. 
Let $T \in \cl(L_2(\Omega))$ be positive and suppose that 
$T L_2(\Omega) \subset C(\overline \Omega)$.
Then for all $\varepsilon > 0$ there exists a 
$K \in C(\overline \Omega \times \overline \Omega)$
such that 
\[
(T^{2+\varepsilon} u)(x)
= \int_\Omega K(x,y) \, u(y) \, dy
\]
for all $u \in L_2(\Omega)$ and $x \in \overline \Omega$.
\end{thm}
\begin{proof}
The operator $T$ is compact from $L_2(\Omega)$ into $L_2(\Omega)$
by Corollary~\ref{ckernel232}.
Hence $T^{\varepsilon / 2}$ is compact from $L_2(\Omega)$ into $L_2(\Omega)$.
The closed graph theorem implies that the 
operator $T$ is bounded from $L_2(\Omega)$ into $C(\overline \Omega)$.
Therefore the operator $T^{1+\varepsilon / 2}$ is compact 
from $L_2(\Omega)$ into $C(\overline \Omega)$.
Since $T^{2+\varepsilon} = (T^{1+\varepsilon / 2}) (T^{1+\varepsilon / 2})^*$
the result follows from Proposition~\ref{pkernel201}\ref{pkernel201-2}.
\end{proof}

In Example~\ref{xkernel203} we constructed a positive operator $T$ on a bounded open 
set $\Omega \subset \Ri^d$ such that $T L_2(\Omega) \subset C(\overline \Omega)$,
but $T^2$ does not have a kernel in $C(\overline \Omega \times \overline \Omega)$.
Hence the power $2+\varepsilon$ in Theorem~\ref{tkernel204.5} is optimal.

In the situation of Theorem~\ref{tkernel204.5} one has in
general $T^\alpha L_2(\Omega) \not\subset C(\overline \Omega)$ for all 
$\alpha \in (0,1)$.
We show this by an example.

\begin{exam} \label{xkernel404}
Let $\Omega = (0,2\pi)$ and for all $n \in \Zi$ define $e_n \in C(\overline \Omega)$
by $e_n(x) = e^{i n x}$.
Then $(e_n)_{n \in \Zi}$ is an orthonormal basis for $L_2(\Omega)$.
For all $n \in \Ni$ let $\lambda_n \in [0,\infty)$ be such that 
$\sum_{n=1}^\infty \lambda_n < \infty$, but 
$\sum_{n=1}^\infty \lambda_n^\alpha = \infty$
for all $\alpha \in (0,1)$ (such a sequence exists).
Define $K \in C(\overline \Omega \times \overline \Omega)$ by 
\[
K(x,y)
= \sum_{n=1}^\infty \lambda_n \, e_n(x) \, \overline{ e_n(y) }
.  \]
Let $S \in \cl(L_2(\Omega))$ be the associated operator.
Then $S$ is positive.
Define $T = S^{1/2}$.
One deduces from Theorem~\ref{tkernel310} that 
$T^{1/2} L_2(\Omega) \subset C(\overline \Omega)$.

Let $\alpha \in (0,1)$ and suppose that 
$T^\alpha L_2(\Omega) \subset C(\overline \Omega)$.
Then $T^\alpha$ is Hilbert--Schmidt by Corollary~\ref{ckernel232},
so $S^\alpha = (T^{\alpha})^2$ is trace class
and $\sum_{n=1}^\infty \lambda_n^\alpha < \infty$.
This is a contradiction.
\end{exam}

\section{A variation of Mercer's theorem} \label{Skernel5}

In Mercer's theorem a continuous kernel $K$ is given on 
$\overline \Omega \times \overline \Omega$,
where $\Omega$ is bounded.
In this section we wish to consider continuous kernels which may be
merely defined on $\Omega \times \Omega$.
If they are Hilbert--Schmidt then we investigate the 
associated operator~$T$.
A central role is played again by the condition 
$T L_2(\Omega) \subset C(\overline \Omega)$.
For continuous kernels the inclusion can be characterised in terms of the 
kernel what we do in the next lemma.
In fact, we want to be slightly more general.

Recall, if $\Omega \subset \Ri$ is open, non-empty and $X \subset \Ri^d$ is a set such that 
$\Omega \subset X \subset \overline \Omega$, then we provide $C(X)$ 
with the Fr\'echet topology of uniform convergence on compact subsets of~$X$.
We emphasise that $\Omega$ does not need to be bounded.

\begin{lemma} \label{lkernel501}
Let $\Omega \subset \Ri$ be open, non-empty and $X \subset \Ri^d$ a set such that 
$\Omega \subset X \subset \overline \Omega$.
Let $K \in C(X \times \Omega)$ and $T \in \cl(L_2(\Omega))$.
Suppose that for all $u \in C_c(\Omega)$ the equality
\[
(T u)(x) = \int_\Omega K(x,y) \, u(y) \, dy
\]
is valid for almost every $x \in \Omega$.
Then the following are equivalent.
\begin{tabeleq}
\item \label{lkernel501-1}
$T L_2(\Omega) \subset C(X)$.
\item \label{lkernel501-2}
$\sup_{x \in F} \int_\Omega |K(x,y)|^2 \, dy < \infty$ for every compact $F \subset X$.
\end{tabeleq}
\end{lemma}
\begin{proof}
`\ref{lkernel501-1}$\Rightarrow$\ref{lkernel501-2}'.
Let $F \subset X$ be compact.
Then the operator $u \mapsto (T u)|_F \in C(F)$ is continuous by the closed graph theorem.
Hence there exists a $c > 0$ such that 
\[
|(K(x,\cdot), \overline u)_{L_2(\Omega)}|
= |(T u)(x)| 
\leq c \, \|u\|_{L_2(\Omega)}
\]
for all $u \in C_c(\Omega)$ and $x \in F$.
Then $\|K(x,\cdot)\|_{L_2(\Omega)} \leq c$ for all $x \in F$ and the implication follows.

`\ref{lkernel501-2}$\Rightarrow$\ref{lkernel501-1}'.
If $u \in C_c(\Omega)$, then the continuity of $T u$ follows from the Lebesgue dominated 
convergence theorem.
Next, let $u \in L_2(\Omega)$ and let $x,x_1,x_2,\ldots \in X$ with $\lim_{n \to \infty} x_n = x$.
Choose $F = \{ x, x_1,x_2,\ldots \} $.
Then $F$ is compact.
So by assumption there exists a $c > 0$ such that 
$\int_\Omega |K(z,y)|^2 \, dy \leq c^2$ for all $z \in F$.
Let $\varepsilon > 0$.
There exists a $v \in C_c(\Omega)$ such that $\|u - v\|_{L_2(\Omega)} \leq \varepsilon$.
Then 
\begin{eqnarray*}
\lefteqn{
|(T u)(x_n) - (T u)(x)| 
} \hspace*{5mm}  \\*
& \leq & |(T u)(x_n) - (T v)(x_n)| 
   + |(T v)(x_n) - (T v)(x)|
   + |(T v)(x) - (T u)(x)|  \\
& \leq & c \, \|u-v\|_{L_2(\Omega)} + |(T v)(x_n) - (T v)(x)| + c \, \|u-v\|_{L_2(\Omega)}
\leq |(T v)(x_n) - (T v)(x)| + 2 c \, \varepsilon
\end{eqnarray*}
for all $n \in \Ni$.
Hence $\limsup_{n \to \infty} |(T u)(x_n) - (T u)(x)| \leq 2 c \, \varepsilon$
and $\lim_{n \to \infty} (T u)(x_n) = (T u)(x)$.
\end{proof}

The main theorem of this section is as follows.
Note that Mercer's theorem is a special case if one chooses $\Omega$ bounded 
and $X = \overline \Omega$.

\begin{thm} \label{tkernel502}
Let $\Omega \subset \Ri$ be open, non-empty and $X \subset \Ri^d$ a set such that 
$\Omega \subset X \subset \overline \Omega$.
Let $K \in C(X \times X)$.
Further, let $T \in \cl(L_2(\Omega))$ be a compact positive operator 
such that $T L_2(\Omega) \subset C(X)$
and such that 
\[
(T u)(x) 
= \int_\Omega K(x,y) \, u(y) \, dy
\]
for all $x \in \Omega$ and $u \in C_c(\Omega)$.

Then there exist an orthonormal basis 
 $(e_n)_{n \in \Ni}$ in $L_2(\Omega)$ 
and for all $n \in \Ni$ there is a $\lambda_n \in [0,\infty)$ such that $\lambda_n \, e_n \in C(X)$ 
and $T e_n = \lambda_n \, e_n$ for all $n \in \Ni$.
In particular, $e_n \in C(X)$ if $\lambda_n \neq 0$, and 
$\lambda_n \, e_n \otimes \overline{e_n} \in C(X \times X)$
for all $n \in \Ni$.
Moreover, 
\[
K(x,y)
= \sum_{n=1}^\infty \lambda_n \, e_n(x) \, \overline{ e_n(y) }
\]
for all $x,y \in X$ and the series
$\sum \lambda_n \, |e_n \otimes \overline{e_n}|$ converges
uniformly on compact subsets of $X \times X$.

Finally, $K(x,x) \geq 0$ for all $x \in X$
and 
\[
\sum_{n=1}^\infty \lambda_n = \int_\Omega K(x,x) \, dx
.  \]
In particular, $T$ is trace class if and only if $\int_\Omega K(x,x) \, dx < \infty$.
\end{thm}
\begin{proof}
Since $T$ is a compact positive operator
there exists an orthonormal basis $(e_n)_{n \in \Ni}$ 
for $L_2(\Omega)$ of eigenfunctions of $T$.
For all $n \in \Ni$ let $\lambda_n \in [0,\infty)$ be such that 
$T e_n = \lambda_n \, e_n$.
Then $\lambda_n \, e_n = T e_n \in C(X)$
for all $n \in \Ni$, and in particular $e_n \in C(X)$ if $\lambda_n \neq 0$.
For all $N \in \Ni$ define $K_N \in C(X \times X)$
by 
\[
K_N(x,y) = \sum_{n = 1}^N \lambda_n \, e_n(x) \, \overline{ e_n(y) }
\]
and define $T_N \in \cl(L_2(\Omega))$ by 
\[
(T_N u)(x) = \int_\Omega K_N(x,y) \, u(y) \, dy
.  \]
Then $T_N u = \sum_{n=1}^N \lambda_n \, (u,e_n) \, e_n$
and $(T - T_N)(u) = \sum_{n=N+1}^\infty \lambda_n \, (u,e_n) \, e_n$
for all $u \in L_2(\Omega)$.
Hence $T - T_N$ is positive.
If $x \in \Omega$, then
\[
(K - K_N)(x,x)
= \lim_{r \downarrow 0} 
   \frac{1}{|B(x,r)|^2} \, ((T - T_N) \one_{B(x,r)}, \one_{B(x,r)})_{L_2(\Omega)}
\in [0,\infty)
\]
and therefore $K_N(x,x) \leq K(x,x)$.
By continuity 
\[
0 \leq 
\sum_{n=1}^N \lambda_n \, |e_n(x)|^2
= K_N(x,x) 
\leq K(x,x) 
\]
for all $x \in X$.
So the series $\sum \lambda_n \, |e_n|^2$ is pointwise convergent
and $\sum_{n=1}^\infty \lambda_n \, |e_n(x)|^2 \leq K(x,x)$ for all $x \in X$.

If $x,y \in X$, then 
\begin{eqnarray*}
\sum_{n=1}^\infty \lambda_n \, |e_n(x) \, \overline{ e_n(y) }|
& \leq & \Big( \sum_{n=1}^\infty \lambda_n \, |e_n(x)|^2 \Big)^{1/2}
     \Big( \sum_{n=1}^\infty \lambda_n \, |e_n(y)|^2 \Big)^{1/2}  \\
& \leq & K(x,x)^{1/2} \, K(y,y)^{1/2}
< \infty
.  
\end{eqnarray*}
Define $\widetilde K \colon X \times X \to \Ci$ by 
\[
\widetilde K(x,y)
= \sum_{n=1}^\infty \lambda_n \, e_n(x) \, \overline{ e_n(y) }
.  \]
Then $|\widetilde K(x,y)| \leq K(x,x)^{1/2} \, K(y,y)^{1/2}$ for all 
$(x,y) \in X \times X$, so $\widetilde K$ is bounded on compact subsets
of $X \times X$.
It will take quite some effort to show that $\widetilde K = K$.

Let $x \in X$.
Let $F \subset X$ be compact.
We shall show that the series $\sum \lambda_n \, e_n(x) \, \overline{e_n}$
converges uniformly on~$F$.
Let $\varepsilon > 0$.
There exists an $N \in \Ni$ such that 
$\sum_{n=N}^\infty \lambda_n \, |e_n(x)|^2 < \varepsilon^2$.
Then 
\[
\sum_{n=N}^\infty \Big| \lambda_n \, e_n(x) \, \overline{e_n(y)} \Big|
\leq \Big( \sum_{n=N}^\infty \lambda_n \, |e_n(x)|^2 \Big)^{1/2}
     \Big( \sum_{n=N}^\infty \lambda_n \, |e_n(y)|^2 \Big)^{1/2}
\leq \varepsilon \, K(y,y)^{1/2}
\]
for all $y \in F$.
So the series $\sum \lambda_n \, e_n(x) \, \overline{e_n}$
converges uniformly on $F$.
Consequently the function $\widetilde K(x,\cdot)$ is continuous on $F$ and then
also on $X$.
Similarly, the function $\widetilde K(\cdot,y)$ is continuous for all $y \in X$.
Therefore $\widetilde K$ is separately continuous.

Let $u \in C_c(\Omega)$.
Then for all $x \in X$ the series $\sum \lambda_n \, e_n(x) \, \overline{e_n} \, u$
is uniformly convergent on $\supp u$.
Hence 
\[
\int_\Omega \widetilde K(x,y) \, u(y) \, dy
= \sum_{n=1}^\infty \int_\Omega \lambda_n \, e_n(x) \, u(y) \, \overline{e_n(y)} \, dy
= \sum_{n=1}^\infty \lambda_n \, (u,e_n)_{L_2(\Omega)} \, e_n(x) 
\]
for all $x \in X$.
On the other hand, 
$T u = \sum_{n=1}^\infty \lambda_n \, (u,e_n)_{L_2(\Omega)} \, e_n$ in $L_2(\Omega)$,
so $(T u)(x) = \sum_{n=1}^\infty \lambda_n \, (u,e_n)_{L_2(\Omega)} \, e_n(x)$ 
for almost every $x \in \Omega$.
Also $(Tu)(x) = \int_\Omega K(x,y) \, u(y) \, dy$ for all $x \in X$.
Therefore 
\begin{equation}
\int_\Omega \widetilde K(x,y) \, u(y) \, dy
= \int_\Omega K(x,y) \, u(y) \, dy
\label{etkernel502;2}
\end{equation}
for almost every $x \in \Omega$.
Since $\widetilde K$ is bounded on compact subsets of $X \times X$ it follows
from the Lebesgue dominated convergence theorem that 
$x \mapsto \int_\Omega \widetilde K(x,y) \, u(y) \, dy$ is continuous on~$X$.
Hence (\ref{etkernel502;2}) is valid for all $x \in X$.
Now let $x \in X$.
Then (\ref{etkernel502;2}) implies that $\widetilde K(x,\cdot) = K(x,\cdot)$
almost everywhere on $\Omega$.
So by continuity one concludes that $\widetilde K(x,\cdot) = K(x,\cdot)$
pointwise on $X$, that is $\widetilde K(x,y) = K(x,y)$ for all $y \in X$.
Hence $K = \widetilde K$.

We proved that 
$\sum_{n=1}^\infty \lambda_n \, |e_n(x)|^2
= \widetilde K(x,x)
= K(x,x)$
for all $x \in X$.
So $\sum_{n=1}^\infty \lambda_n \, |e_n|^2$ is continuous by the 
assumption that $K$ is continuous.

Let $F \subset X$ be compact.
Then by Dini's theorem the series $\sum \lambda_n \, |e_n|^2$
converges uniformly on $F$.
Since 
\[
\sum_{n=N}^\infty \Big| \lambda_n \, e_n(x) \, \overline{e_n(y)} \Big|
\leq \Big( \sum_{n=N}^\infty \lambda_n \, |e_n(x)|^2 \Big)^{1/2}
     \Big( \sum_{n=1}^\infty \lambda_n \, |e_n(y)|^2 \Big)^{1/2}
\]
for all $x \in F$ and $y \in X$, the series 
$\sum \lambda_n \, |e_n \otimes \overline{e_n}|$
is uniformly convergent on compact subsets of $X \times X$.

Finally, the monotone convergence theorem gives
\[
\int_\Omega K(x,x) \, dx
= \int_\Omega \sum_{n=1}^\infty \lambda_n \, |e_n|^2
= \sum_{n=1}^\infty \int_\Omega \lambda_n \, |e_n|^2
= \sum_{n=1}^\infty \lambda_n
.  \]
So the operator $T$ is trace class if and only if $\int_\Omega K(x,x) \, dx < \infty$.
\end{proof}

\begin{cor} \label{ckernel504}
Let $\Omega \subset \Ri$ be open, non-empty and $X \subset \Ri^d$ a set such that 
$\Omega \subset X \subset \overline \Omega$.
Let $K \in C(X \times X)$
and suppose that $K|_{\Omega \times \Omega} \in L_2(\Omega \times \Omega)$.
Let $T \in \cl(L_2(\Omega))$ be the Hilbert--Schmidt operator
with kernel $K|_{\Omega \times \Omega}$. 
Suppose that $T$ is positive and $T L_2(\Omega) \subset C(X)$.

Then all conclusions of Theorem~\ref{tkernel502} are valid.
\end{cor}
\begin{proof}
Every Hilbert--Schmidt operator is compact.
Let $u \in C_c(X)$.
Then $(T u)(x) = \int_\Omega K(x,y) \, u(y) \, dy$ for almost 
every $x \in \Omega$, since $T$ is the Hilbert--Schmidt operator
with kernel $K|_{\Omega \times \Omega}$. 
But $T u$ is continuous on $\Omega$ by assumption and also 
$x \mapsto \int_\Omega K(x,y) \, u(y) \, dy$ is continuous by the Lebesgue 
dominated convergence theorem.
Therefore $(T u)(x) = \int_\Omega K(x,y) \, u(y) \, dy$ for all $x \in \Omega$
and the conditions of Theorem~\ref{tkernel502} are satisfied.
\end{proof}

\begin{cor} \label{ckernel503}
Let $\Omega \subset \Ri$ be open and bounded.
Let $K \in C_b(\Omega \times \Omega)$ and 
let $T \in \cl(L_2(\Omega))$ be the Hilbert--Schmidt operator
with kernel $K$. 
Suppose that $T$ is positive.
Then $T$ is trace class.
\end{cor}

We remark that the positivity of $T$ can be characterised if the kernel is continuous,
as is well known (cf.\ \cite{BCR} Chapter~3, Exercise~1.24).

\begin{lemma} \label{lkernel506}
Let $\Omega \subset \Ri$ be open,
let $K \in C(\Omega \times \Omega)$ and $T \in \cl(L_2(\Omega))$.
Suppose that for all $u \in C_c(\Omega)$ the equality
\[
(T u)(x) = \int_\Omega K(x,y) \, u(y) \, dy
\]
is valid for almost every $x \in \Omega$.
Then the following are equivalent.
\begin{tabeleq}
\item \label{lkernel506-1}
$T$ is positive.
\item \label{lkernel506-2}
$\sum_{k,l=1}^N c_k \, \overline{c_l} \, K(x_k,x_l) \geq 0$
for all $N \in \Ni$, $x_1,\ldots,x_N \in \Omega$ and $c_1,\ldots,c_N \in \Ci$.
\end{tabeleq}
\end{lemma}

Our arguments in the proof of Theorem~\ref{tkernel502} stem from the 
classical result where $K \in C(\overline \Omega \times \overline \Omega)$
and $\Omega$ is bounded, see for example Werner \cite{Werner} Satz~VI.4.2.
Theorem~\ref{tkernel502} is covered by \cite{SunMercer} Theorem~2, where 
reproducing kernel Hilbert spaces are used for the proof of identity~(\ref{etkernel502;2}),
but the trace formula is missing.
Ferreira, Menegatto and Oliveira use the same arguments as we for proving 
\cite{FerreiraMenegattoOliveira} Theorem~2.6, but the statement is different.

\section{Examples} \label{Skernel6}

In this section we give examples which illustrate our results.
Let $\Omega \subset \Ri^d$ be open connected and bounded.
Depending on the problem one might obtain kernels in 
$C(X \times X)$ for different choices of $X$ with 
$\Omega \subset X \subset \overline \Omega$.

The first example is with Neumann boundary conditions.

\begin{exam} \label{xkernel601}
Let $\Omega \subset \Ri^d$ be an open connected bounded set 
with continuous boundary.
Further let $\Gamma \subset \partial \Omega$ be a relatively open set
such that for all $z \in \Gamma$ there is an $r > 0$ such that 
$B(z,r) \cap \Gamma$ is a Lipschitz graph with $B(z,r) \cap \Omega$
on one side.
Consider the Neumann Laplacian $\Delta^N$ in $L_2(\Omega)$ and let 
$S$ be the $C_0$-semigroup generated by $\Delta^N$.
Choose $X = \Omega \cup \Gamma$.
Then $S_t$ is self-adjoint and $S_t L_2(\Omega) \subset C(X)$ 
for all $t > 0$ by \cite{ERe2} Lemmas~5.1 and~6.7.

Since $\Omega$ has continuous boundary, the operator $\Delta^N$ has compact resolvent.
Denote by $0 = \lambda_1 < \lambda_2 \leq \lambda_3 \leq \ldots$ the 
eigenvalues of $- \Delta^N$ repeated with multiplicity and 
by $(e_n)_{n \in \Ni}$ an orthonormal basis for $L_2(\Omega)$
satisfying $- \Delta^N e_n = \lambda_n \, e_n$ for all $n \in \Ni$.
We may choose $e_1 = \one_{\overline \Omega}$.
Then $e_n \in C(X)$ for all $n \in \Ni$.
It follows from Corollary~\ref{ckernel204.4} that 
$S_t = (S_{t/3})^3$ has a kernel $K_t \in C(X \times X)$ for all $t > 0$.
Moreover, Theorem~\ref{tkernel502} gives
\[
K_t(x,y) 
= \sum_{n=1}^\infty e^{-\lambda_n t} e_n(x) \, \overline{e_n(y)}
\]
for all $t > 0$ and $x,y \in X$.
Furthermore, for all $t > 0$ the series $\sum e^{-\lambda_n t} |e_n \otimes \overline{e_n}|$
converges uniformly on compact subsets of $X \times X$.
\end{exam}

\begin{exam} \label{xkernel602}
Let $\Omega \subset \Ri^d$ be an open connected bounded set.
Consider the Dirichlet Laplacian $\Delta^D$ in $L_2(\Omega)$.
Let $\Gamma$ be the set of all regular points in the sense of Wiener.
Choose $X = \Omega \cup \Gamma$.

We first show that 
\begin{equation}
\{ u \in D(\Delta^D) : \Delta^D u \in L_\infty(\Omega) \}
\subset C(X) .
\label{exkernel602;1}
\end{equation}
Let $u \in D(\Delta^D)$ and suppose that $f = \Delta^D u \in L_\infty(\Omega)$.
Clearly $u \in C(\Omega)$ by elliptic regularity. 
It sufficies to show that $\lim_{x \to z, \; x \in \Omega} u(x) = 0$
for all $z \in \Gamma$.
Denote by $E_d$ the Newtonian potential on $\Ri^d$ and write 
$w = E_d * \tilde f$, where $\tilde f \in L_\infty(\Ri^d)$ is the extension
of $f$ by~$0$.
Then $w \in C(\Ri^d) \cap H^1(\Ri^d)$.
Moreover, $\Delta(w|_\Omega) = f = \Delta u \in H^{-1}(\Omega)$ as distributions.
Write $h = w|_\Omega - u$ and $\varphi = w|_{\partial \Omega}$.
Then $h$ is the Perron solution of $\varphi$ by \cite{AD2} Theorem~1.1.
In particular $\lim_{x \to z, \; x \in \Omega} h(x) = \varphi(z)$
for all $z \in \Gamma$ since $z$ is a regular point.
Because $\varphi(z) = w(z)$ this implies that 
$\lim_{x \to z, \; x \in \Omega} u(x) = 0$ as required
and (\ref{exkernel602;1}) follows.

Let $S$ be the $C_0$-semigroup generated by $\Delta^D$.
Then $S_t$ is self-adjoint for all $t > 0$.
Let $t > 0$ and $u \in L_2(\Omega)$.
Then $S_t u \in D(\Delta_D)$ by holomorphy of the semigroup $S$ and 
$\Delta^D \, S_t u = S_{t/2} \Delta^D \, S_{t/2} u \in L_\infty(\Omega)$ by 
ultracontractivity of $S$.
Hence $S_t u \in C(X)$ by (\ref{exkernel602;1}).
Since $\Omega$ is bounded, the operator $\Delta^D$ has compact resolvent.
Denote by $0 < \lambda_1 < \lambda_2 \leq \lambda_3 \leq \ldots$ the 
eigenvalues of $- \Delta^D$ repeated with multiplicity and 
by $(e_n)_{n \in \Ni}$ an orthonormal basis for $L_2(\Omega)$
satisfying $- \Delta^D e_n = \lambda_n \, e_n$ for all $n \in \Ni$.
As in the previous example one deduces that $e_n \in C(X)$ for all $n \in \Ni$
and the operator
$S_t = (S_{t/3})^3$ has a kernel $K_t \in C(X \times X)$ for all $t > 0$.
Moreover, Theorem~\ref{tkernel502} implies that 
\[
K_t(x,y) 
= \sum_{n=1}^\infty e^{-\lambda_n t} e_n(x) \, \overline{e_n(y)}
\]
for all $t > 0$ and $x,y \in X$.
Finally, for all $t > 0$ the series $\sum e^{-\lambda_n t} |e_n \otimes \overline{e_n}|$
converges uniformly on compact subsets of $X \times X$.
\end{exam}

\subsection*{Acknowledgements}
The first-named author is most grateful for the hospitality extended
to him during a fruitful stay at the University of Auckland and the 
second-named author
for a wonderful stay at the University of Ulm.
This work is supported by the Marsden Fund Council from Government funding,
administered by the Royal Society of New Zealand.

\end{document}